\documentclass[11pt]{article}
\usepackage{indentfirst}
\usepackage{mathrsfs}
\usepackage{latexsym,lineno}
\usepackage{algorithm}
\usepackage{enumerate}
\usepackage{algorithmic}
\usepackage{epsfig}
\usepackage{color}
\usepackage{amsmath}\usepackage{fleqn}\usepackage{verbatim}
\usepackage{epsf}
\usepackage{amsthm}\usepackage{graphicx, float}\usepackage{graphicx}
\usepackage{amsfonts}
\usepackage{amssymb}\usepackage{graphpap}
\usepackage{epic}\usepackage{curves}
\usepackage{tikz}
\usepackage{subfigure}
\usepackage{mathrsfs}
\usepackage{comment}
\usepackage{pstricks}
\usepackage{hyperref}

\usepackage[numbers]{natbib}
\renewcommand{\newblock}{\\} 
\providecommand{\bibinfo}[2]{#2} 

\newtheorem{theorem}{Theorem}[section]

\newtheorem{lemma}[theorem]{Lemma}

\newtheorem{claim}[theorem]{Claim}

\title{\bf On generalized Tur{\'a}n problems with bounded matching number and circumference\thanks{Research was partially supported by the National
Nature Science Foundation of China (grant numbers 12331012)}}
\date{\today}
\author {Yongchun Lu$^{1}$, \, Liying Kang$^{1,3}$\thanks{\em Corresponding author. Email address: lykang@shu.edu.cn (L. Kang)}, \, Yisai Xue$^{2}$\\
	{\small $^{1}$ Department of Mathematics, Shanghai University,
		Shanghai 200444, P.R. China}\\
	{\small$^{2}$School of Mathematics and Statistics, Ningbo University, Ningbo, China}\\
	{\small$^{3}$Newtouch Center for Mathematics of Shanghai University,
		Shanghai,  China, 200444}}

\date{}

\begin{document}
	
	\maketitle
	
	\begin{abstract}
		Let \( \mathcal{F} \) be a family of graphs. The generalized Tur\'an number \( \operatorname{ex}(n, K_r, \mathcal{F}) \) is the maximum number of $K_r$ in an \( n \)-vertex graph that does not contain any member of \( \mathcal{F} \) as a subgraph.  Recently, Alon and Frankl initiated the study of Tur\'an problems with bounded matching number. In this paper, we determine the generalized Tur\'an number of \( C_{\geq k} \) with bounded matching number.
	\end{abstract}
	\begin{flushleft}
		\hspace{2.5em}\textbf{Keywords:} generalized Tur\'an number, matching number, circumference\\
		\hspace{2.5em}\textbf{AMS (2000) subject classification:} 05C35
	\end{flushleft}

	\section{Introduction}
	The study of extremal graph theory has been a cornerstone in combinatorial mathematics, focusing on the properties of graphs that extremize certain parameters while adhering to given constraints.
	A central theme within this field is the exploration of Tur\'an numbers, which quantify the maximum number of edges in a graph that avoids containing specific substructures as subgraphs.
	This concept was pioneered by Tur\'an's theorem, which determined $\operatorname{ex}(n,K_{k+1})$, the maximum number of edges in a graph with bounded clique number.
	Erd\H{o}s and Gallai \cite{gallai1959maximal} further expanded this domain by determining $\operatorname{ex}(n,M_{s+1})$, the maximum number of edges in a graph with a bounded matching number.
	
	Let $T$ be a fixed graph and $\mathcal{F}$ be a family of graphs. A graph $G$ is called $\mathcal{F}$-{\sl free} if $G$ does not contain any copy of the graphs in $\mathcal{F}$. We denote by $\mathcal{N}(T,G)$ the number of copies of $T$ in $G$. The {\sl generalized Tur\'an number} of $\mathcal{F}$ is defined as follows:
	\begin{align}
		\operatorname{ex}(n,K_r,\mathcal{F})=\max\{\mathcal{N}(K_r,G)|G\text{ is an $n$-vertex }\mathcal{F}\text{-free graph}\}.\nonumber
	\end{align}
	We call the $n$-vertex $\mathcal{F}$-free graph attaining $\operatorname{ex}(n,K_r,\mathcal{F})$ copies of $K_r$ as the {\sl extremal graph} of $\mathcal{F}$. When $T=K_2$, it is the classical Tur\'{a}n number $\operatorname{ex}(n,\mathcal{F})$. The concept of the generalized Tur\'an number was formally introduced by Alon and Shikhelman \cite{alon2016many} in 2016, and Wang \cite{wang2020shifting} further studied the generalized Tur\'an number of matchings.
	
	\begin{theorem}[\cite{wang2020shifting}]\label{wang}
		For any $s \geq 2$ and $n \geq 2 k+1 $, we have \\
		$$\operatorname{ex}\left(n, K_s, M_{k+1}\right)=\max \left\{\binom{2k+1}{s},\binom{k}{s}+(n-k)\binom{k}{s-1}\right\}.$$
	\end{theorem}

	In 2022, Alon and Frankl \cite{alon2024turan} determined the exact value of $\operatorname{ex}(n,\{K_{k+1},M_{s+1}\})$.
	
	\begin{theorem}[\cite{alon2024turan}]
		For $n\geq 2s+1$ and $k\geq 2$, $\operatorname{ex}(n,\{K_{k+1},M_{s+1}\})=\max\{e(T_k(2s+1)),e(G(n,k))\}$ where $G(n,k)=T_{k-1}(s)\vee{I_{n-s}}$.
	\end{theorem}
	
	Following this breakthrough, many  relevant results have been published.
	Given a positive integer $n$ and a graph $F$, Gerbner \cite{gerbner2024turan} considered
	$\operatorname{ex}(n, \{F, M_{s+1}\})$ in general, and determined its value apart from a constant additive term.
	
	\begin{theorem}[\cite{gerbner2024turan}]
		If $\chi(F)>2$ and $n$ is sufficiently large, then $\operatorname{ex}\left(n,\left\{F, M_{s+1}\right\}\right)=$ $\operatorname{ex}(s, \mathcal{F})+s(n-s)$, where $\mathcal{F}$ is the family of graphs obtained by deleting an  independent set from $F$.
	\end{theorem}
	
	Ma and Hou \cite{ma2023generalizedturanproblembounded}  determined the exact value of  $\operatorname{ex}(n,K_k, \{K_{k+1}, M_{s+1}\})$ and gave an asymptotic value of $\operatorname{ex}(n,K_k, \{F, M_{s+1}\})$
	for general $F$ with an error term $O(1)$.
	Zhu and Chen \cite{zhu2024extremal} determined $\operatorname{ex}(n, K_r, \{F,M_{s+1}\})$
	when $F$ is color critical with $\chi(F)\geq\max\{r+1,4\}$.
	Gerbner \cite{Gerbner2023}  extended these investigations by replacing $K_r$ with an arbitrary graph $H$.
	
	Recently, Xue and Kang \cite{xue2024generalized}  investigated  the generalized Tur\'an problem of  matchings and paths for any sufficiently large $n$. Apart from matching, Tur{\'a}n problems concerning the circumference are also prominent topics in extremal graph theory. Following the literature, we denote by $C_{\geq k}$ the family of cycles with length at least $k$. The exact value of $\operatorname{ex}\left(n, C_{\geq k}\right)$ was determined by Woodall \cite{woodall1976maximal} and independently by Kopylov \cite{kopylov1977maximal}.
	During the last few years, Chakraborti and Chen \cite{chakraborti2024exact} investigated the generalized Tur\'an number of $C_{\geq k}$.  Very recently, Dou, Ning and Peng \cite{dou2024number} determined the generalized Tur\'an number with bounded  clique number and circumference.
Zhao and Lu \cite{zhao2024generalizedturanproblemsmatching} determined $\operatorname{ex}(n, K_r, \{C_{\geq 2k+1}, M_{s+1}\})$ when $s \geq 2k+1$ and $k \geq r-1$, and  $\operatorname{ex}(n, K_r, \{C_{\geq 2k}, M_{s+1}\})$ when $k \geq r$.
Motivated by these results, we determined the value of  $\operatorname{ex}(n, K_r, \{C_{\geq k},M_{s+1}\})$ for all $s$, $r$ and $k$. One can refer to references  \cite{chvatal1976degrees, katona2024extremal, liu2024extremal, luo2018maximum, wang2023spectral, zhang2023maximum} for more information on related topics.
	
	Set \( p := \left\lfloor\frac{k-1}{2}\right\rfloor + 1 \geq3.\) If $n$ is sufficiently large, one can easily check that
	$$\operatorname{ex}(n, K_r, \left\{C_{\geq k}, M_{s+1}\right\}) \leq \operatorname{ex}\left(n, K_r, M_{s+1}\right) = \mathcal{N}(K_r, (K_s \vee I_{n-s})).$$
	Note that \( K_s \vee I_{n-s} \) is \( C_{\geq k} \)-free if $p>s$. Then
	$$
	\operatorname{ex}\left(n, K_r, \left\{C_{\geq k}, M_{s+1}\right\}\right) = \binom{s}{r} + (n-s)\binom{s}{r-1}.
	$$
	Thus we only need to consider the case \( p \leq s \).

	We give the exact values of the generalized Tur\'an number of $\{C_{\geq k}, M_{s+1}\}$ by considering the parity of $k$. The constructions of
	$G_1,G_2,G_3,G_4,G_5$ and $G_6$ will be described in Section \ref{section 3} and Section \ref{section 4}.
	\begin{theorem}\label{odd}
		Let  \( s \geq p \geq 3 \), and $n$ be a sufficiently large integer. Assume $s-p+1=a(p-2)+b$, where \( 0 \leq b \leq p-3 \).
		\begin{itemize}
			\item If $0 \le b <\left\lceil\frac{p-1}{2}\right\rceil$, then $$\operatorname{ex}\left(n, K_r,\left\{C_{\geq 2p-1}, M_{s+1}\right\}\right)=\mathcal{N}\left(K_r, G_1\right).$$
			\item If $\lceil\frac{p-1}{2} \rceil\leq b \leq p-3$, then $$\operatorname{ex}\left(n, K_r,\{C_{\geq 2p-1}, M_{s+1}\right)=\max\{\mathcal{N}(K_r,G_1), \mathcal{N}\left(K_r, G_2\right)\}.$$
		\end{itemize}
	\end{theorem}

	\begin{theorem}\label{even}
		Let  \( s \geq p \geq 3 \), and $n$ be a sufficiently large integer.  Assume $s-p+1=c (p-1)+d$, where $0 \leq d \leq p-2$.
		\begin{itemize}
			\item If $d=0$, then
			$$
			\operatorname{ex}\left(n, K_r,\left\{C_{\geq 2 p}, M_{s+1}\right\}\right)=\max \left\{\mathcal{N}\left(K_r, G_3\right), \mathcal{N}\left(K_r, G_4\right)\right\}.
			$$
			\item If $1 \leq d \leq p-3$, then
			$$
			\operatorname{ex}\left(n, K_r,\left\{C_{\geq 2 p}, M_{s+1}\right\}\right)=\max \left\{\mathcal{N}\left(K_r, G_3\right), \mathcal{N}\left(K_r, G_4\right), \mathcal{N}\left(K_r, G_5\right), \mathcal{N}\left(K_r, G_6\right)\right\}.
			$$
			\item If $d=p-2$, then
			$$
			\operatorname{ex}\left(n, K_r,\left\{C_{\geq 2 p}, M_{s+1}\right\}\right)=\max \left\{\mathcal{N}\left(K_r, G_3\right), \mathcal{N}\left(K_r, G_4\right), \mathcal{N}\left(K_r, G_6\right)\right\}.
			$$
		\end{itemize}
	\end{theorem}

	\section{Preliminaries}
	
	Let $G=(V,E)$ be a simple graph with vertex set $V=V(G)$ and edge set $E=E(G)$. Set $e(G)=|E(G)|.$
	For $S\subseteq V(G)$, denote by $G[S]$ the graph induced by $S$, and denote by $G\setminus S$ the graph obtained from $G$ by deleting all vertices of $S$ and all edges incident with $S$.  For $H\subseteq G$, let $G\setminus H$ denote the graph obtained from $G$ by removing all edges in $E(H)$, and subsequently removing all isolated vertices in $H$. For two vertex disjoint graphs $G$ and $H$, we write \(G\cup H\) as the {\sl union} of $G$ and $H$. We write $k$ disjoint copies of $H$ as $kH$. The {\sl join} of $G$ and $H$, denoted by $G \vee H$, is the graph obtained from \(G \cup H\) by adding all possible edges between $G$ and $H$.
	For a subgraph \( H \subseteq G \), denote the neighborhood of \( v \) in \( H \) by \( N_H(v) := V(H) \cap N_G(v) \).
	Moreover, we define $N_G(S)=\{v \text{ }|\text{ }\exists \text{ }u\in S, uv\in E(G)\}$ and $N_H(S) = \{ v \in V(H) \mid \exists u \in S, \ uv \in E(G) \}.$ For subsets \( V_1, V_2 \subseteq V(G) \), \( E_G(V_1, V_2) \) denotes the set of edges between \( V_1 \) and \( V_2 \) in \( G \). When there is no ambiguity, we omit the subscript \( G \). Additionally, let $G[V_1, V_2]$ denote the subgraph induced by the edge set $E_G(V_1,V_2).$
	
	To {\sl identify} nonadjacent vertices $u$ and $v$ of a graph $G$ is to replace $u,v$ by a single vertex $w,$ and each edge $f\in E(G)$ that incident with $u$ or $v$ is replaced by an edge incident with $w.$ To {\sl contract} an edge $e=uv$ is to delete the edge and then identify its ends. The resulting graph is denoted by $G/uv.$
	Let $I_n$ be an independent set of size $n$. For a matching \( M \) in a graph \( G \), we say \( M \) is a {\sl near-perfect matching} if \( M \) covers all but one vertex of \( G \).
	
	We introduce some lemmas which will be used in our proofs.
	\begin{lemma}[\cite{dirac1952some}]\label{Dirac}
		Let $G$ be a connected graph. If $P$ is a longest path of $G$ with ends $u$ and $v$, then
		$$
		|V(P)| \geq \min \{|V(G)|, d(u)+d(v)+1\} .
		$$
	\end{lemma}
	\begin{lemma}[\cite{kopylov1977maximal}]\label{Kopylov}
		Let G be a 2-connected n-vertex graph with a path $P$ of $m$ edges with ends $x$ and $y$. For $v \in V(G)$, let $d_P(v)=|N(v) \cap V(P)|$. Then $G$ contains a cycle of length at least $\min \left\{m+1, d_P(x)+d_P(y)\right\}$.
	\end{lemma}
	
	\begin{lemma}[\cite{chakraborti2021many}]\label{ch}
		Let $r, w, x, y$, and $z$ be non-negative integers such that $r \geq 2$, $x+y=w+z, x \geq w, x \geq z$, and $x \geq r$. Then,
		$$
		\binom{x}{r}+\binom{y}{r} \geq\binom{ w}{r}+\binom{z}{r} .
		$$
		Moreover, the inequality is strict if $x>w$ and $x>z$.
	\end{lemma}
	
	We use \( \mathcal{G}_{\text{tree}}(F_1, F_2, \dots, F_\ell) \) to denote the family of connected graphs $G$ if the blocks of $G$ are $F_1, F_2, \dots, F_\ell$.  We refer to the graph \( G \) in \( \mathcal{G}_{\text{tree}}(F_1, F_2, \dots, F_\ell) \) as the {\sl block-cut tree} with blocks \( F_1, F_2, \dots, F_\ell .\) Additionally, we call $G\in \mathcal{G}_{\text{tree}}(F_1, F_2, \dots, F_\ell)$ a  {\sl block-cut star}  if all blocks $F_1, F_2, \dots, F_\ell$ share exactly one common vertex.
	In a block-cut tree \( G \in \mathcal{G}_{\text{tree}}(F_1, F_2, \dots, F_\ell) \), a block is called a {\sl block-leaf} if the block contains exactly one cut-vertex of $G$. A block-cut tree \( G \in  \mathcal{G}_{\text{tree}}(F_1, F_2, \dots, F_\ell) \) is defined to be {\sl strict} if  each \( F_i \) is a  2-connected graph. The common vertex in a block-cut star is referred to as the {\sl center vertex}. Since every connected graph can be decomposed into a block-cut tree,  if a block-cut tree is not strict, it must contain a cut-edge as a block.
	
	\begin{lemma}\label{near-perfect matching}
		Let \( \ell \) be a positive integer. For a  strict block-cut tree \( G\in \mathcal{G}_{\text{tree}}(F_1, F_2, \dots, F_\ell) \), if every \( F_i, i\in[\ell]\), is a Hamiltonian graph with odd order. Then for any vertex \( v \in V(G) \), there exists a near-perfect matching \( M \) of \( G\) such that \( v \notin V(M) \).
	\end{lemma}
	
	\begin{proof}
		We prove the statement by induction on \( \ell \). For \( \ell = 1 \), \( G \) has a Hamiltonian cycle with odd order, and the result holds trivially. For $\ell\geq 2$, without loss of generality, let \( F_\ell \) be a block-leaf of \( G\). Let \( G' := G \setminus F_\ell \), and $u$ be the cut-vertex  between \( V(F_\ell) \) and \( V(G') \). For any vertex \( v \in V(G) \), we consider the following two cases.
		If \( v \in V(G') \), by the induction hypothesis, we can always find a near-perfect matching \( M_1\) in \( G' \) excluding \( v \). Recall that \( F_\ell \) is a Hamiltonian graph with odd order.  We can find a perfect matching $M_2$ in \(F_\ell \setminus \{u\} \). Clearly, \( M_1 \cup M_2 \) is a near-perfect matching of \( G \) excluding \( v \). If \( v \in V(F_\ell) \), one can  find a near-perfect matching \( M_3 \) in \( F_\ell \) excluding $v$ and a perfect matching \( M_4 \)  in \( V(G') \setminus \{u\} \). Then, $M_3\cup M_4$ is a near-perfect matching of $G$ excluding $v$. The proof is complete.
	\end{proof}
	
	\begin{lemma}\label{star}
		Let \( G \) be a strict block-cut tree and \( G^* \) be a block-cut star in \( \mathcal{G}_{\text{tree}}(F_1, F_2, \dots,F_\ell)\). If every block \( F_i \) is a Hamiltonian graph, then \( \nu(G) \geq \nu(G^*) \).
	\end{lemma}
	
	\begin{proof}
		Let \( q \) be the number of  Hamiltonian graphs with even order in \( \{F_1, F_2, \dots, F_\ell\} \).
		When \( q = 0 \), it follows from Lemma \ref{near-perfect matching} that \(\nu(G)=\nu (G^*)\).
		If \( q \geq 1 \),  let \( t(G) \) and \( t(G^*) \) be the number of unmatched vertices in \( G \) and \( G^* \), respectively. To show \(\nu(G)\geq\nu (G^*)\), it suffices to prove that \(t(G)\le t (G^*)\).
		It is easy to see that \( t({G^*}) = q - 1 \). Next we prove that \( t({G}) \leq q - 1 \) by induction on $\ell$. For \( \ell = 1 \), since \( G \) is a  Hamiltonian graph with even order, we have \( t(G) = 0 \). The result holds.

		For \(\ell\geq 2 \), if there is a block-leaf with odd order in \( G\), without loss of generality, assume the block-leaf is \( F_\ell \). Let \( G' := G \setminus F_\ell \) and \( u \) be the cut-vertex between \( G' \) and \( F_\ell \). Since \( F_\ell \) is Hamiltonian, it is easy to see that there exists a maximum matching of \( G \) that matches all vertices in \( F_\ell \setminus\{u\} \). Therefore, \( t(G) = t(G') \leq q - 1 \) by the induction hypothesis.
		
		If all block-leaves in \( G \) have even order, note that \( G \) has at least two block-leaves. Thus for a block-leaf $F_\ell$,  \( G' = G \setminus F_\ell \) has \( q' := q - 1 \geq 1 \) blocks and each block is a  Hamiltonian graph with even order. Since \( F_\ell \) is Hamiltonian, it is easy to see that there exists a maximum matching of \( G \) that unmatched at most one vertex in \( F_\ell \setminus\{ u \}\). Thus, \( t(G) \leq t(G') + 1 \leq (q' - 1) + 1 = q - 1 \) by the induction hypothesis.
		The proof is complete.
	\end{proof}

	\begin{lemma}\label{contract}
		Let G be a graph and $uv\in E(G)$. If $G$ is $\left\{C_{\geq k}, M_{s+1}\right\}$-free, then $G/uv$ is also $\left\{C_{\geq k}, M_{s+1}\right\}$-free.
	\end{lemma}
	
	\begin{proof}
		Obviously, \( G / uv \) is \( M_{s+1} \)-free. We now show that \( G / uv \) is \( C_{\geq k} \)-free. Denote the new vertex in \( G / uv \) by \( w \). Suppose for the sake of contradiction, there exists  a cycle in \( C_{\geq k} \) in $G/uv$, which is referred to  as \( Q \).
		Obviously,  $Q$ contains  vertex $w$ and two edges, say $wu_1, wv_1$.  Then neither both $u_1, v_1$ are  adjacent to $u$ nor both $u_1, v_1$ are  adjacent to $v$. Otherwise, $G$ contains  a cycle in $C_{\geq k}$. Without loss of generality, assume $u_1$ is adjacent to $u$, $v_1$ is adjacent to $v$. By replacing \( u_1wv_1 \) with \( u_1uvv_1 \), we
		find  a cycle in \( C_{\geq k} \)
		in $G$, a contradiction.
	\end{proof}
	To give the main proofs, we need the following key lemma.
	\begin{lemma}\label{stability}
		Let \( s \), $k$ be two integers and $p=\lfloor\frac{k-1}{2}\rfloor+1$. For any sufficiently large \( n \) and $p\leq s$, there exist an integer \(t_0\leq \binom{2s}{p}p+ 2s +1-p\) and a \(\{C_{\geq k}, M_{s+1}\} \)-free graph \( H \) on \( n \) vertices with \(\mathcal{N}(K_r, H) = \operatorname{ex}(n, K_r, \{C_{\geq k}, M_{s+1}\} )\) and a partition \( V(H) = X \cup Y \cup Z \) that satisfies the following:
		\begin{itemize}
			\item[\rm (1)] \( H[X] = K_{p-1} \);
			\item[\rm (2)] \( Y \) is an independent set with \( |Y| = n - t_0 - p + 1 \) and each vertex in \( Y \) has the neighborhood \( X \);
			\item[\rm (3)] every vertex in \( Z \) has a degree of at least \( p \), and all its neighbors are contained in \( X \cup Z \).
		\end{itemize}
	\end{lemma}
	\begin{proof}
		Let \( G \) be a \(\{C_{\geq k}, M_{s+1}\}\)-free graph with \( \mathcal{N}(K_r, G) = \text{ex}(n, K_r, \{C_{\geq k}, M_{s+1}\}) \). The assumption  \( p \leq s \)  implies that \( K_{p-1} \vee I_{n-p+1} \) is \(\{C_{\geq k}, M_{s+1}\}\)-free.
		Hence,
		\begin{align}
			\mathcal{N}(K_r, G) &\geq \binom{p-1}{r-1}(n - p + 1) + \binom{p-1}{r}.\label{1}
		\end{align}
		
		Let \( U \) be the set of  vertices matched by a maximum matching of \( G \).
		Then \( |U| \leq 2s \), and \( V(G) \setminus U \) forms an independent set. Define \( L := \{ v \in V(G) \setminus U : d(v) \geq p \} \). One can check that \( |L| \leq p \binom{2s}{p } \). If not, note that each vertex in \( L \) has a neighborhood in \( U \) with order at least \( p \) and \( U \) has at most \( \binom{2 s}{p} \) such subsets. By the Pigeonhole Principle, there must exist \( p \) vertices of \( L \) sharing at least \( p \) common neighbors in \( U \), leading to a cycle in \( C_{\geq k} \), a contradiction.
		
		Define \( W = V(G) \setminus (U \cup L) \). Then every vertex in \( W \) has degree at most \( p - 1 \) and
		\begin{align}
			|W| \geq n - 2s - \binom{2s}{p}p.\label{2}
		\end{align}
		Now, let \( W^{\prime} \subseteq W \) be the set of vertices whose neighborhood induces a \( (p-1) \)-clique. Then, for any vertex \( v \in W \backslash W^{\prime} \), the number of \( r \)-cliques containing \( v \) is at most \( \binom{p-1}{r-1} - 1 \). So we have
		\begin{align}
			\mathcal{N}(K_r, G)
			&\leq \left|W'\right| \binom{p-1}{r-1} + \left|W \setminus W'\right| \left(\binom{p-1}{r-1} - 1\right) + \binom{n - |W|}{r}.\label{3}
		\end{align}
		By  \eqref{1} and \eqref{3}, we get
		\begin{align*}
			\binom{p-1}{r-1}(n-p+1)+\binom{p-1}{r} \leq\left|W'\right|\binom{p-1}{r-1}+|W\setminus W'|\left(\binom{p-1}{r-1}-1\right)+\binom{n-|W|}{r}.
		\end{align*}
		Combining with \eqref{2}, we  have
		\begin{align*}
			\binom{p-1}{r-1}|W'| & \geq(n-p+1)\binom{p-1}{r-1}+\binom{p-1}{r}-n\left(\binom{p-1}{r-1}-1\right)-\binom{2 s+p\binom{2 s}{p}}{r} \\
			& \geq n-(p-1)\binom{p-1}{r-1}+\binom{p-1}{r}-\binom{2 s+p\binom{2 s}{p}}{r} \\
			& \geq \frac{n}{2} \quad \text{(as $n$ is sufficiently large)}.
		\end{align*}
		Thus there exists an integer $c_0>\max\{k+1, 2s+3\}$ such that  \( \left|W^{\prime}\right| \geq c_0 \binom{2s}{p-1} \) as \( n \) is sufficiently large. Since  $\text{the number of } (p-1) \text{-sets in } U \text{ is at most } \binom{2 s}{p-1}$, and  \( \left|W'\right| \geq c_0 \binom{2s}{p-1} \), by Pigeonhole
		Principle, there is at least one $(p-1)$-set $X$ in $U$, which is the neighborhood of at least $c_0$ vertices in $W'.$ Recall the definition of $W^{\prime}$, we have $G[X] = K_{p-1}$.
		
		If there exists a vertex \( v \) in \( V(G) \) with degree at most \( p - 1 \), we replace its neighborhood with \( X \), creating a new graph \( G^{\prime} \). We claim that \( G^{\prime} \) is \( \{C_{\geq k}, M_{s+1}\} \)-free. If \( G^{\prime} \) contains a cycle of length at least $k$ or a copy of \( M_{s+1} \), denoted by \( Q \), then \( Q \) must contain the vertex \( v \). Since \( |N(X)| \geq c_0 \), we can always find a vertex \(v'\in\bigcap_{x\in X}N_G(x) \setminus V(Q) \),  and then replace vertex \( v \) by \( v^{\prime} \). Since \( v \) and \( v^{\prime} \) share identical neighborhoods, \( G \) contains a copy of $Q$, a contradiction.  Moreover, such processes do not decrease the number of the $K_r$.
		Now, we keep repeating this process until no vertex in the graph has a degree less than $p-1$, and all vertices with degree $p-1$ share the same neighborhood $X$. The resultant graph is denoted as $H$. Then $H$ is $\left\{C_{\geq k}, M_{s+1}\right\}$-free and $\mathcal{N}\left(K_r, H\right)=\operatorname{ex}\left(n, K_r,\left\{C_{\geq k}, M_{s+1}\right\}\right)$. Note that $H[X]=K_{p-1}$. Let $Y$ denote the set of vertices in $H$ with a degree of $p-1$, then $|Y|\geq |W|\geq  n - 2s - \binom{2s}{p}p$. Finally, we define $Z=V(H) \backslash(X \cup Y)$, then every vertex in $Z$ has a degree of at least $p$ and  $t_0=|Z|\leq  \binom{2s}{p}p+ 2s +1-p$.
	\end{proof}

	\section{Forbidding $C_{\geq 2p-1}$}\label{section 3}
	In this section, we determine the generalized Tur\'an number of $\{C_{\geq 2p-1}, M_{s+1}\}$.
	We first construct $G_1$ and $G_2$ as follows.\\
	\textbf{Construction.} Let  \( s \geq p \geq 3 \), and $n$ be a  sufficiently large integer. Assume $s-p+1=a(p-2)+b$, where \( 0 \leq b \leq p-3 \). Define
	\begin{itemize}
		\item
		$G_1=K_1 \vee \left( K_{p-2} \vee I_{n-p+1-a(2p-3)} \cup a K_{2p-3} \right),$
		\item
		$G_2=K_1 \vee \left( K_{p-2} \vee I_{n-p-a(2p-3)-2b} \cup a K_{2p-3} \cup K_{2b+1} \right).$
	\end{itemize}
	Obviously, $G_1, G_2$ are $\{C_{\geq 2p-1}, M_{s+1}\}$-free.
	The lower bounds of $\operatorname{ex}(n, K_r, \{C_{\geq k}, M_{s+1}\} )$ are established by considering graphs $G_1$ and $G_2$. For the upper bound, let \( \mathcal{G} \) be the family of extremal graphs with the properties stated in Lemma \ref{stability}. For each $G \in \mathcal{G}$, denote by $X_G, Y_G$ and $Z_G$ the vertex sets described in Lemma \ref{stability}. In cases where there is no ambiguity, we omit the subscript $G$. In the following, we always set $X=\left\{v_1, v_2, \ldots, v_{p-1}\right\}$.
	
	Define $\Phi: \mathcal{G} \rightarrow \mathbb{R}^3$ as a map such that $\Phi(G)=\left(e(G),k_3(G),c\left(Z_G\right)+\left|Y_G\right|\right)$, where $k_3(G)$ denotes the number of $3$-cliques in $G$ and $c\left(Z_G\right)$ denotes the number of connected components in $G[Z]$. For $G, G^{\prime} \in \mathcal{G}$, we say $\Phi\left(G^{\prime}\right)>_{\text{lex}}\Phi(G)$ if $G^{\prime}$ has a larger lexicographical order than $G$. We choose the extremal graph $G\in \mathcal{G}$ such that $\Phi(G)$ is lexicographically maximal, and under this condition, $G$ has the largest maximum degree.
	
	\begin{claim}\label{find cycle}
		If there is a path \( P \) in \( G[Z \cup X] \) such that the ends of \( P \) are contained in \( X \) and \( |V(P) \cap Z| \geq |V(P) \cap X |=2\), then 	there exists a cycle of length at least \(2p-1 \) in \( G.\)
	\end{claim}
	
	\begin{proof}
		Assume  the ends of $P$ are $v_1$ and $v_2$. Since $|V(P) \cap Z| \geq |V(P) \cap X |=2$, we have $|V(P)|\ge 4$. As $G[X,Y]$ is a complete bipartite graph, there exists a path $P^{\prime}$ from $v_1$ to $v_2$ in $G[X,Y]$ such that $|V(P')|=2p-3$. Clearly, the concatenation of \( P \) and \( P^{\prime} \) forms a cycle of length at least \( 2p-1 \). This completes the proof.
	\end{proof}
	
	It follows from Lemma \ref{stability} (3) that each connected component in \( G[Z] \) contains at least two vertices.  Let $H$ be a connected component in $G[Z]$, we will determine the structure of \( H\).
	\begin{claim}\label{base}
		Let $H$ be a connected component of $G[Z]$. Then the following statements hold:
		\begin{itemize}
			\item[\rm (1)] $|N_{G[X]}(V(H))|=1$;
			\item[\rm (2)] $H$ is a strict block-cut tree.
		\end{itemize}
	\end{claim}
	\begin{proof}
		We prove (1) by contradiction. Suppose that $|N_{G[X]}(V(H))|\neq1$, then $|N_{G[X]}(V(H))|=0$ or $|N_{G[X]}(V(H))|\geq 2$. If $|N_{G[X]}(V(H))|=0$,  we can add an edge to \( E(V(H), X) \), thereby obtaining a graph \( G' \). Clearly, \( G' \) is \(\{C_{\geq 2p-1}, M_{s+1}\}\)-free and \(\mathcal{N}(K_r, G^{\prime}) \geq \mathcal{N}(K_r, G)\) holds. However, $e(G')>e(G)$ contradicts the assumption of $G$. So $|N_{G[X]}(V(H))|\geq 2$.
		
		If there exist at least two vertices \( u_1, u_2 \in V(H) \) which are connected to distinct vertices in \( G[X] \) respectively, without loss of generality, assume \( u_1 v_1, u_2 v_2 \in E(G) \).
		Since \( H \) is connected, there exists a path \( P' \) in \( H \) with ends \( u_1 \) and \( u_2 \). By adding the edges \( u_1 v_1 \) and \( u_2 v_2 \), we can extend the path \( P' \)  to a path \( P \) with ends \( v_1 \) and \( v_2 \).
		Clearly, $|V(P) \cap V(H)| \geq |V(P) \cap X| = 2.$ By Claim~\ref{find cycle},  there exists a cycle of length at least \( 2p-1 \), leading to a contradiction. So we may assume there is only one vertex (denoted by \( u \)) in \( H \) which  is adjacent to vertices in \( G[X] \).
		Without loss of generality, let $uv_1\in E(G)$. In this case, we contract the edge $uv_1$, losing at most \( \binom{p-1}{r-1} \) \( r \)-cliques. Next we add a vertex to \( Y \), obtaining  \( \binom{p-1}{r-1} \) \( r \)-cliques. Let the resulting graph  be \( G' \).  By Lemma \ref{contract}, \( G' \) is \( \{C_{\geq 2p-1}, M_{s+1}\} \)-free, and \( \mathcal{N}(K_r, G') \geq \mathcal{N}(K_r, G) \). It is easily  checked that  \( \Phi(G') >_{\text{lex}} \Phi(G) \), leading to a contradiction.

		To show statement (2), it is sufficient to prove that there is no cut-edge in $H$. If there exists a cut-edge \( e\) in \( H \). We contract the edge \( e \), and add a vertex to \( Y \) to obtain the resultant graph \( G' \). By Lemma \ref{contract}, \( G' \) is also \(\{C_{\geq 2p-1}, M_{s+1}\}\)-free.
		Statement (1) implies that contracting the cut-edge \( e \) removes at most 2 edges and 1 triangle, without affecting any larger cliques. By adding a new vertex to \( Y \), we add at least \( p-1 \geq 2 \) edges and at least 1 triangle, which implies that \( \mathcal{N}(K_r, G') \geq \mathcal{N}(K_r, G) \). It is easily  checked that  \( \Phi(G') >_{\text{lex}} \Phi(G) \), leading to a contradiction.
	\end{proof}
	
	By Claim~\ref{base} (1), we may suppose that \( v_1 \) is the unique vertex in \( X \) which is adjacent to vertices in \( H .\)  Let $B_1, B_2, \ldots, B_\ell$ be the   blocks of $H$.
	\begin{claim}\label{block size}
		For any $i \in [\ell]$, the following statements are true.
		\begin{itemize}
			\item [\rm (1)]$d_H\left(v_1\right) \geq p$;
			\item [\rm (2)]For any vertex $u \in V\left(B_i\right)$,  $d_{B_i}(u) \geq p-1$;
			\item [\rm (3)]For any vertex \( u \in V(B_i)\setminus N_G(v_1) \),  \( d_{B_i}(u) \geq p \).
		\end{itemize}
	\end{claim}
	\begin{proof}
		We prove all statements by contradiction.
		For statement (1), suppose to the contrary that $ d_H(v_1) \leq p-1$,  we contract an edge $e_1\in E(\{v_1\},V(H))$, losing at most  \( \binom{p-1}{r-1} \) $r$-cliques. Then, we add a vertex to $Y$, obtaining  \( \binom{p-1}{r-1} \) $r$-cliques. Let the resulting graph be $G'$. Clearly, we have  \( \mathcal{N}(K_r, G') \geq \mathcal{N}(K_r, G) \). By Lemma \ref{contract},  $G'$ is also  $\{C_{\geq 2p-1}, M_{s+1}\}$-free.
		If there exists a vertex in $H$  whose  degree is less than \( p \) after contracting edge $e_1$, we replace its neighborhood with \( X \). Note that this process does not decrease the number of \( r \)-cliques. We repeatedly apply this process, which will eventually terminate. Let the final graph be \( G'' \). Through this process, \( G'' \) remains \( \{C_{\geq 2p-1}, M_{s+1}\} \)-free, and we have \( \mathcal{N}(K_r, G'') \geq \mathcal{N}(K_r, G') \geq \mathcal{N}(K_r, G) \) and \( \Phi_1(G'') >_{\text{lex}} \Phi_1(G) \), which contradicts the assumption of \( G \).
		
		For the second statement, suppose there exist  an integer $i_0\in [\ell]$ and  a vertex \( u \in V(B_{i_0})  \)  such that \( d_{B_{i_0}}(u) \leq p-2 \). By Lemma \ref{stability} (3) and Claim \ref{base} (1), \( u \) must be a cut-vertex in \( H \).
		We contract an edge $e_2\in E(\{u\}, V(B_{i_0}))$, losing at most \( \binom{p-2}{r-1} + \binom{p-2}{r-2} \) \( r \)-cliques. Then, we add a vertex to \( Y \), thereby adding at least \( \binom{p-1}{r-1} \) \( r \)-cliques. Let \( G' \) be the resulting graph. It follows from \( \binom{p-1}{r-1}=\binom{p-2}{r-1} + \binom{p-2}{r-2} \) that \( \mathcal{N}(K_r, G') = \mathcal{N}(K_r, G) \). By Lemma \ref{contract},  \( G'\) is \( \{C_{\geq 2p-1}, M_{s+1}\} \)-free.
		If there exists a vertex in $H$ whose  degree is less than \( p \) after contracting $e_2$, we replace its neighborhood with \( X \). Repeat this process until it terminates. Let \( G'' \) be the final graph. Clearly, \( G'' \) remains \( \{C_{\geq 2p-1}, M_{s+1}\} \)-free and \( \mathcal{N}(K_r, G'') \geq \mathcal{N}(K_r, G') \geq \mathcal{N}(K_r, G) \). Then we have   \( \Phi(G'') >_{\text{lex}} \Phi(G) \), a contradiction.
		
		For the last statement, suppose there exists an integer $i_0\in [\ell]$ and  a vertex \( u \in V(B_{i_0})\setminus N_G(v_1)  \)  such that  \( d_{B_{i_0}}(u) \leq p-1 \). Clearly, \( u \) is a cut-vertex in \( H \).
		We contract an edge $e_2\in E(\{u\}, V(B_{i_0}))$, losing at most \( \binom{p-1}{r-1} \) \( r \)-cliques. Then we add a vertex to \( Y \), adding at least \( \binom{p-1}{r-1} \) \( r \)-cliques. Let \( G' \) be the resulting graph. Obviously, \( G'\) is \( \{C_{\geq 2p-1}, M_{s+1}\} \)-free and  \( \mathcal{N}(K_r, G') \geq \mathcal{N}(K_r, G) \).
		If there is a vertex in \( G' \) whose degree is less than \( p \) after contracting $e_2$, we replace its neighborhood with \( X \). Repeat this process until it terminates, Let the final graph be \( G'' \). Then  \( G'' \) remains \( \{C_{\geq 2p-1}, M_{s+1}\} \)-free and \( \mathcal{N}(K_r, G'') \geq \mathcal{N}(K_r, G') \geq \mathcal{N}(K_r, G) \). One can check that \( \Phi(G'') >_{\text{lex}} \Phi(G) \), a contradiction.
	\end{proof}

	\begin{claim}\label{core}
		For any $1\leq i\neq j\leq \ell$, there are no vertices $u\in V(B_i)\setminus V(B_j)$ and  $v\in V(B_j)\setminus V(B_i)$ such that both $u$ and $v$ are adjacent to $v_1$.
	\end{claim}
	\begin{proof}
		Suppose to the contrary that there are two vertices $u\in V(B_1)\setminus V(B_2)$ and  $v\in V(B_2)\setminus V(B_1)$ such that $uv_1\in E(G)$ and $vv_1\in E(G)$. Let $B$ be the maximal $2$-connected subgraph of $G[V(H)\cup\{v_1\}]$ containing $v_1$. Then $B$ contains $B_1,B_2$ as subgraphs and $v_1\in V(B)$. Since \( B\) is maximal,  every vertex in \( V(H) \setminus V(B) \) is not adjacent to \( v_1 \).
		Combining with Claim \ref{block size}, we have $\delta(B)\geq p$ and $|V(B)|\geq 2p$. By Lemma \ref{Dirac}, there exists a longest path $P$ in $H$ of length at least $\min\{|V(B)|,2p+1 \}\geq 2p.$ Assume $u$ and $v$ are the end vertices of $P$. Since $P$ is the longest path in $B$, we have $d_P(u)\geq p$ and $d_P(v)\geq p$. Using Lemma \ref{Kopylov}, one can find a cycle of length at least $\min\{|V(P)|, 2p\}\geq 2p$ in $H$, leading to  a contradiction.
	\end{proof}
	
	Let $H'=G[V(H)\cup \{v_1\}]$.  The next claim follows immediately from Claim \ref{block size} and Claim \ref{core}.
	\begin{claim}\label{min deg}
		\( H' \) is a strict block-cut tree, and each vertex \( u \) in \( H' \) has a degree of at least \( p \) within the block  containing \( u .\)
	\end{claim}
	
	We denote \( B'_1,B'_2, \ldots, B'_\ell\) as the blocks in the strict block-cut tree \( H' \).
	\begin{claim}\label{H}
		For any  $i \in [\ell]$, \( B_i' \) is a Hamiltonian graph with  $p+1\leq |V(B_i')|\leq 2p-2$.
	\end{claim}
	\begin{proof} By Claim \ref{min deg}, $|V(B_i')|\geq p+1$ for any  $i \in [\ell]$.
		We choose the longest path \( P\) in \( B_i'\), and let $u,v$ be two ends of $P.$ Since $P$ is the longest path in $B_i'$,  $d_{P}(u)\geq p$ and $d_{P}(v)\geq p$.
		By Lemma \ref{Kopylov}, \( B_i'\) contains a cycle $C$ with $|V(C)|\geq \min \left\{|V(P)|, 2p \right\}= |V(P)|$. Applying Lemma \ref{Dirac} to the longest path $P$,  we deduce that
		$$
		|V(B_i')| \geq |V(C)|\geq |V(P)| \geq \min \{|V(B_i')|, 2p+1\} = |V(B_i')|,
		$$
		which implies that \( B_i' \) contains a Hamiltonian cycle $C$. Moreover, we have $|V(B_i')|\leq 2p-2$ as \( B_i' \) is $C_{\geq 2p-1}$-free.
	\end{proof}
	
	\begin{claim}\label{clique}
		$H'$ is a clique with  $p+1\leq |V(H')|\leq 2p-2$.
	\end{claim}
	\begin{proof}
		We first prove that  $H'$ is a block. If $H'$ is not a block, then there are at least two blocks in \( H' \). We transform \(H' \) into a block-cut star \( S'\) with \( v_1 \) as its center vertex, where \( S' \) has the same blocks as \( H' \). Let \( G' \) be the resultant graph. Clearly, \( G' \) is \( C_{\geq 2p-1} \)-free. We now show that \( G' \) is also \( M_{s+1} \)-free. This can be deduced by considering two cases.
		
		Case 1. If all blocks in $H'$ are odd. By Lemma \ref{near-perfect matching}, one can find maximum matchings in \( H'\) and \( S' \), both of which are near-perfect matchings that exclude the vertex \( v_1 \). It follows that
		$$
		\nu(G) = \nu(H') + \nu(G \setminus H') = \nu(S') + \nu(G' \setminus S') = \nu(G').
		$$
		
		Case 2. If there exists an even block in \( H' \), then \( S' \)  contains an even Hamilton graph.
		Clearly, the center vertex $v_1$ must be contained in every maximum matching of $S'$ and
		also in every maximum matching of $G'\setminus S'$. It follows that
		$$
		\nu(G') = \nu(S') + \nu(G' \setminus S') - 1.
		$$
		By Lemma \ref{star}, we get
		\begin{align*}
			\nu(G) \geq \nu(H') + \nu(G \setminus H') - 1 \geq \nu(S') + \nu(G \setminus H') - 1 = \nu(S') + \nu(G' \setminus S') - 1 = \nu(G').
		\end{align*}
		So \( G' \) is \( M_{s+1} \)-free. It is clear that \( \mathcal{N}(K_r, G') = \mathcal{N}(K_r, G) \).  Note that we split a connected component \(H \) into at least two connected components in $G[Z]$.  Then \( \Phi(G') >_{\text{lex}} \Phi(G) \),  a contradiction. Thus $H'$ is a block.
		
		By Claim \ref{H},  we have $p+1\leq |V(H')|\leq 2p-2$. Moreover, one can add edges to \( H' \) to make it to be a complete graph while keeping \( \{C_{\geq 2p - 1}, M_{s+1}\} \)-free. Recall that \( G \) is an extremal graph with the maximum number of edges, so \( H' \) must be a clique.
	\end{proof}
	
	Recall that  $\Phi(G)$ is lexicographically maximal,  and \( G \) has the largest maximum degree in \( \mathcal{G} \). Then   all blocks in $G[Z]$ intersect at one vertex. In the following, we assume that the  vertex is $v_1$.
	
	\begin{claim}\label{structure}
		There is at most one block with order less than \( 2p - 3 \) in  $G[Z].$
	\end{claim}
	\begin{proof}
		We prove it by contradiction. Suppose otherwise, by Claim \ref{clique}, there exist two maximal cliques \( K_x \) and \( K_y \) in $G[Z]$, where \( p  \leq y \leq x \leq 2p - 4 \).
		We first deduce that neither \( K_x \) nor \( K_y \) is an even clique. Suppose there exists an even clique between $K_x$ and $K_y$. Then we replace $K_x\cup K_y$ with $K_{x+1}\cup K_{y-1}$ to obtain a graph $G'.$
		The operation does not increase $\nu(G)$ or create a cycle in $C_{\geq2p-1}.$ By Lemma \ref{ch}, we have
		$$
		\binom{x+2}{r} + \binom{y}{r} \geq \binom{x+1}{r} + \binom{y+1}{r}.
		$$
		Then \( \mathcal{N}(K_r, G') \geq \mathcal{N}(K_r, G) \) and $e(G')>e(G)$,  a contradiction. Thus \( K_x \) and \( K_y \) are odd cliques.
		
		We replace $K_x\cup K_y$ with $K_{x+2}\cup K_{y-2}$. It is easy to verify that this operation does not   increase $\nu(G)$ or create a cycle in $C_{\geq2p-1}.$ Let \( G' \) denote the resulting graph. According to  Lemma \ref{ch},  we deduce
		$$
		\binom{x+3}{r} + \binom{y-1}{r} \geq \binom{x+1}{r} + \binom{y+1}{r},
		$$
		which ensures that \( \mathcal{N}(K_r, G') \geq \mathcal{N}(K_r, G) \) and $e(G')>e(G)$. Thus $\Phi(G') >_{\text{lex}} \Phi(G),$ a contradiction.
	\end{proof}
	
	Clearly, the matching number in \( G[Z] \) is at most \( s - p + 1 \). Define \( a = \lfloor\frac{s-p+1}{p-2}\rfloor \) and \( b = s-p+1-a(p-2) \),  \( 0 \leq b < p-2 \).

	We first show that the number of  \( K_{2p-3} \) in \( G[Z] \) is \( a \).
	If the number of \( K_{2p-3} \) in \( G\left[Z \right] \) is smaller than $a$,  one can check
	\begin{align*}
		\mathcal{N}\left(K_r, G\right) \leq\mathcal{N}\left(K_r,K_1\vee\left(K_{p-2} \vee I_{n-p+1-a(2p-3) }\cup aK_{2p-3}\right)\right)=\mathcal{N}(K_r,G_1)
	\end{align*}
	for any $2\leq r\leq 2p-2.$ In particular, we have $\mathcal{N}(K_3, G_1)>\mathcal{N}(K_3, G)$, which leads to a contradiction. Hence the number of  \( K_{2p-3} \) in \( G[Z] \) must be \( a \).
	By Claim \ref{structure}, there is at most a block with order less than \( 2p - 3 \) in $G[Z].$ In the following, we will determine the structure of the block that has an order less than $2p-3$ in $G[Z].$
	
	For $b< \lceil\frac{p-1}{2}\rceil$,  we claim there is no  block with order less than \( 2p - 3 \) in  $G[Z].$ Otherwise, let $K_t$ be a maximal clique with $t<2p-3.$ Then, we have $t\leq 2b+1<p.$ Recall every connected component of $G[Z]$ has an order at least $p$, leading to a contradiction. Hence, $G=G_1.$
	
	For $\lceil\frac{p-1}{2}\rceil\leq b\leq 2p-3,$ if there is no block with order less than \( 2p - 3 \) in $G[Z].$ Then $G \cong G_1.$
	If there is a maximal clique $K_t$ with order less than \( 2p - 3 \) in  $G[Z]$, we claim that $t=2b+1$. If $t<2b+1$, we change the neighborhood of a vertex in $Y$ to $V(K_t)\cup\{v_1\}$. Then the resulting graph is $\{C_{\geq 2p-1}, M_{s+1}\}$-free and has more cliques than $G$, leading to a contradiction. In this case, we have $G\cong G_2$. The result follows.

	\section{Forbidding $C_{\geq 2p}$}\label{section 4}
	In this section, we determine the generalized Tur\'an number of \( \{C_{\geq 2p}, M_{s+1}\} \). We first construct graphs $G_3,G_4,G_5$, $G_6.$\\
	\textbf{Construction.} Let \( q=\big\lfloor\frac{s-p+1}{p-2}\big\rfloor \) and \( t=s-p+1-q(p-2) \). Define
	$$
	G_3= \begin{cases}
		K_1\vee\left(K_{p-2} \vee I_{n-p+1-q(2 p-3)}\cup q K_{2 p-3}\right) & \text{ if } t=0 \\
		K_1\vee\left(K_{p-2} \vee I_{n-p-q(2 p-3)-2 t}\cup q K_{2 p-3} \cup K_{2 t+1}\right)& \text{ if } t \neq 0 .
	\end{cases}
	$$
	Let \( c=\big\lfloor\frac{s-p+1}{p-1}\big\rfloor \) and \( d=s-p+1-H(p-1) \). Define
	$$
	G_4= \begin{cases}
		K_1\vee\left(K_{p-2} \vee I_{n-p+1-c(2p-2)}\cup c K_{2 p-2} \right) & \text { if } d=0 \\
		K_1\vee\left(K_{p-2} \vee I_{n-p-c(2p-2)-2d} \cup c K_{2 p-2}\cup K_{2 d+1}\right) & \text { if } d \neq 0 .
	\end{cases}
	$$
	\begin{itemize}
		\item For \( 1 \leq d \leq p-3 \), we define
		$$
		G_5=K_1\vee\left(K_{p-2} \vee I_{n- (c-1)(2p-2)-d} \cup(c-p+d+2) K_{2 p-2}\cup (p-d-1) K_{2 p-3}\right).
		$$
		\item For \( 1 \leq d \leq p-2 \), we define
		$$
		G_6=K_1\vee\left(\left(K_{p-2} \vee\left(K_2 \cup I_{n-p-1-c(2p-2)}\right)\right) \cup c K_{2 p-2}\right) .
		$$
	\end{itemize}
	It is easy to check that $G_3, G_4, G_5$ and $ G_6$ are \( \{C_{\geq 2p}, M_{s+1}\} \)-free.
	The lower bounds of \( \operatorname{ex}\left(n, K_r, \left\{C_{\geq 2 p}, M_{s+1}\right\}\right)\) are established by considering graphs $G_3,G_4,G_5$ and $G_6$.  For the upper bounds, let \( \mathcal{G} \) be the family of extremal graphs with the properties stated in Lemma \ref{stability}. Let $X=\{v_1,v_2,\ldots,v_{p-1}\}.$
	
	Define $\Phi: \mathcal{G} \rightarrow \mathbb{R}^4$ as a map such that $\Phi(G)=\left(e(G), k_3(G),c\left(Z_G\right)+\left|Y_G\right|,c\left(Z_G\right)\right)$. We choose $G \in\mathcal{G}$ such that the $\Phi(G)$ is lexicographically maximal, and under this condition $G$ has the largest maximum degree. We will  show that \( G \) is isomorphic to $G_3,G_4,G_5$ or $G_6$.
	\begin{claim}\label{find cycle2}
		If there is a path \( P \) in \( G[Z \cup X] \) such that the ends of \( P \) are contained in \( X \) and \( |V(P) \cap Z| > |V(P) \cap X |=t\geq 2\),	then there exists a cycle of length at least \(2p \) in \( G .\)
	\end{claim}
	\begin{proof}
		Without loss of generality, assume  $\{v_1,v_2\}\subseteq V(P)\cap X.$ Since $|V(P) \cap Z| > |V(P) \cap X |=t$, we have $|V(P)|\ge 2t+1$. As $G[X,Y]$ is a complete bipartite graph, there exists a path $P^{\prime}$ from $v_1$ to $v_2$  in $G[X,Y]$ such that $|V(P')|=2p-2t+1$ and $V(P')\cap V(P)=\emptyset.$ Clearly, the concatenation of \( P \) and \( P^{\prime} \) forms a cycle of length at least \( 2p\).
	\end{proof}

	Define an edge in \( G[Z] \) as an {\sl exceptional edge} if both of its ends are adjacent to every vertex in  \( X \).
	It follows from Lemma \ref{stability} (3) that  a connected component with two vertices in $G[Z]$ must be an exceptional edge.
	\begin{claim}\label{exceptional}
		The number of exceptional edges in \( G[Z] \) is at most 1.
	\end{claim}
	\begin{proof}
		Suppose to the contrary that  \( x_1y_1 \) and \( x_2y_2 \) are two exceptional edges in \( G[Z] \).
		
		Case 1. If $\lvert\{x_1, y_1\}\cap\{x_2,y_2\}\rvert=0$, then we divide it into two cases.
		If \( p = 3 \), we can find a $6$-cycle    \( v_1x_1y_1v_2x_2y_2v_1 \) in $G$, a contradiction.
		If \( p \geq 4 \), one can find a path  \( v_1x_1y_1v_2x_2y_2v_3 \) in $G[X\cup Z]$. By Claim \ref{find cycle2}, there exists a cycle of length at least \(2p \) in $G$, a contradiction.
		
		Case 2. If $\lvert\{x_1, y_1\}\cap\{x_2,y_2\}\rvert=1$, then there exists a $3$-path $P$ in $G[Z]$. The path $P$ together with $\{v_1, v_2\}$ forms a path satisfying the condition of Claim \ref{find cycle2}, which leads to a cycle of length at least \(2p \) in $G$, a contradiction.
	\end{proof}
	
	Let $H$ be a connected component in $G[Z].$
	\begin{claim}\label{cut-ex}
		If $uv$  is a cut-edge  of $H$, then $uv$ is an exceptional edge.
	\end{claim}
	\begin{proof}  Suppose to the contrary that $uv$ is not an exceptional edge. Then one of $u, v$ is adjacent to at most $p-2$ vertices in $X$.
		We contract  edge \( uv \) and add a vertex to \( Y \). Denote by \( G' \) the resultant graph. Clearly, \( G' \) is \( \{C_{\geq 2p}, M_{s+1}\} \)-free. Note that $uv$ is a cut-edge in $H$. Then there are at most $\binom{p-2}{r-2}$ $r$-cliques containing  $e$. Thus we lose at most \(\binom{p-2}{r-1}+ \binom{p-2}{r-2}=\binom{p-1}{r-1} \) \( r \)-cliques and add \( \binom{p-1}{r-1} \) \( r \)-cliques.
		It follows that \( \mathcal{N}(K_r, G') \geq \mathcal{N}(K_r, G) \) and \( \Phi(G') >_{\text{lex}} \Phi(G) \) contradict the assumption of $G$.
	\end{proof}

	\begin{claim}\label{matching}
		If $uv$ is a cut-edge of $H$, then every maximum matching $M$ of $G$ contains $uv$.
	\end{claim}
	\begin{proof}
		By Claim \ref{cut-ex}, $uv$ is an exceptional edge. We claim that there is no edge between $V(H)\setminus\{u,v\}$ and $X$. Otherwise, by Claim \ref{find cycle2}, one can find a path $P$ such that \( |V(P) \cap Z| > |V(P) \cap X |=2\), leading to a cycle of length at least $2p$ in $G$, a contradiction.

		Let $C_u$ and $C_v$ be two connected components in $H\setminus uv$  containing $u$ and $v$, respectively.
		We first show that both \( C_u \) and \( C_v \) contain a maximum matching excluding $u$ and $v$, respectively.  If not,  we may suppose that every maximum matching in \( C_u \) contains \( u \). Then, we split \( u \) into two vertices \( u_1 \) and \( u_2 \), where \( u_1 \) inherits all neighbors of \( u \) in \( C_u \), and \( u_2 \) inherits all neighbors of \( u \) in \( G \setminus C_u \). Then we identify \( u_1 \) and \( v_1 \). Let \( G' \) denote the resultant graph.
		
		The fact that no edge between $V(H)\setminus\{u,v\}$ and $X$ implies that \( G' \) is  \( C_{\geq 2p} \)-free. Next we show that $G'$ is $M_{s+1}$-free. Note that when we split $u$  into $u_1$ and $u_2$, the matching
		number increases by at most one. When we identify $u_1$ and $v_1$, the matching number decreases by $1$, as $u_1$ and $v_1$  are contained in every maximum matching in $C_u$ and $G[X\cup  Y]$, respectively.
		Thus $G'$ is $M_{s+1}$-free.
		Moreover,  we have \( \mathcal{N}(K_r, G')= \mathcal{N}(K_r, G) \) as $uv$ is a cut-edge, and \( \Phi(G') >_{\text{lex}} \Phi(G) \),  leading to a contradiction. Thus both \( C_u \) and \( C_v \) contain  a maximum matching excluding $u$ and $v$, respectively. Let $M_1', M_2'$ be the maximum matching excluding $u$ and $v$ of \( C_u \) and \( C_v \), respectively.
		
		To show $uv$ is contained in every maximum matching of $G$, it
		is sufficient to show that there is no edge  $ux$ or $vx$  in every maximum matching of $G$, where $x\in V(G)\setminus\{u,v\}.$
		Suppose otherwise, we may assume there exists a maximum matching $M$ of $G$ with $ux\in M$, where $x\in X\cup V(C_u)$. Let $M_1=M\cap E(C_u), M_2=M\cap E(C_v)$.
		If $x\in X$, recall that $G[X, Y]$ is a complete bipartite graph, we can always find a  vertex $y$  in $Y$ which is not matched by $M$.  Replacing $M_1\cup M_2\cup \{xu\}$ with $M_1'\cup M_2'\cup \{uv\}\cup \{xy\}$, we can find a matching of $G$ whose size is larger than $M$, a contradiction. If $x\in C_u$, replacing $M_1\cup M_2$ with $M_1'\cup M_2'\cup \{uv\}$, we also find a matching of $G$ whose size is larger than $M$, a contradiction.
	\end{proof}
	
	In the following discussion, we will determine the structure of \( H\) with $|V(H)|\geq3$.
	\begin{claim}\label{2base}
		If \( H\) is a strict block-cut tree in $G[Z]$, then $|N_{G[X]}(V(H))|=1.$
	\end{claim}
	\begin{proof}
		Clearly, there is at least one edge in $E(V(H),X)$. Suppose to the contrary that   $|N_{G[X]}(V(H))|$ $\not= 1$, then $|N_{G[X]}(V(H))|\geq 2.$
		If there exist at least two vertices \( u_1, u_2 \in V(H) \) that are connected to distinct vertices in \( G[X] \) respectively, without loss of generality, assume  \( u_1 v_1, u_2 v_2 \in E(G) \).
		Since \( H \) is a strict block-cut tree, there exists a path \( P' \)  from $u_1$ to $u_2$ in \( H \) with length at least 3. By adding the edges \( u_1 v_1 \) and \( u_2 v_2 \), we can extend the path \( P' \)  to a path \( P \) with ends \( v_1 \) and \( v_2 \). Note that $|V(P) \cap V(H)| > |V(P) \cap X|=2.$ By Claim~\ref{find cycle2},  there  exists a cycle of length at least \( 2p \) in $G$, leading to a contradiction.
		So we may assume that only one vertex in \( H \)  is connected to  vertices in \( G[X] \).
		We contract an edge $e\in E(V(H),X)$, and add a vertex to \( Y \). Let the resulting graph be \( G' .\) Clearly, \( G' \) is \( \{C_{\geq 2p}, M_{s+1}\} \)-free and
		\( \mathcal{N}(K_r, G') \geq \mathcal{N}(K_r, G) \). Furthermore, we have \( \Phi(G') >_{\text{lex}} \Phi(G) \), leading to a contradiction.
	\end{proof}

	\begin{claim}\label{base2}
		Let $H$ be a connected component in $G[Z]$. Then $H$ is a strict block-cut tree.
	\end{claim}
	\begin{proof}
		It is sufficient  to prove there is no cut-edge in $H.$ Otherwise, let $uv$ be a cut-edge of $H.$ By Claim \ref{exceptional} and Claim \ref{cut-ex}, there is no other cut-edges in $G[Z]$. Then  all connected components in \( G[Z] \) are strict block-cut trees, except for \( H \). We perform the following operations:
		\begin{enumerate}
			\item Contract \( e = uv \) into a vertex \( w \), and remove a vertex from \( Y \);
			\item Add  a single edge \( u'v' \) into \( G[Z] \) and connect \( u' \) and \( v' \) to every vertex in \( X \).
		\end{enumerate}
		Denote the resultant graph by \( G' \). One can verify that \( \mathcal{N}(K_r, G') = \mathcal{N}(K_r, G) \) and \( \Phi(G') >_{\text{lex}} \Phi(G) \).
		Now, we show that \( G' \) is \( \{C_{\geq 2p}, M_{s+1}\} \)-free, which will cause a contradiction.
		By	 Claim \ref{matching}, we get $\nu(G')=\nu(G)$. So \( G' \) is \( M_{s+1} \)-free.
		Next we  show that \( G' \) is \( C_{\geq 2p} \)-free. Otherwise,  there exists a cycle of length at least $2p$ in $G'$ (Denoted by $Q$).
		By Lemma \ref{contract},  \( Q \) must contain at least a vertex from \( \{u', v'\} \).
		After performing the above operations, the strict block-cut trees remains unchanged. By Claim \ref{2base},  \( Q \) does not contain any vertex in the block-cut trees.
		Recall that each vertex in \( V(H) \setminus \{u, v\} \) is not adjacent to any vertex in \( X .\)
		Thus $V(Q)$ is a subset of $X\cup Y\cup \{w, u', v'\}$. It is easily checked that the induced subgraph of $X\cup Y\cup \{w, u', v'\}$ in $G'$ can not form a copy $Q$ in $G'$, leading to a contradiction. Therefore, \( G' \) is \( C_{\geq 2p} \)-free.
		Thus \( G' \) is \( \{C_{\geq 2p}, M_{s+1}\} \)-free and \( \Phi(G') >_{\text{lex}} \Phi(G) \), which leads to a contradiction.
	\end{proof}

	Let \( H\) be a connected component in \( G[Z] \) with \( |V(H)| \geq 3 \), and  \( B_1, B_2, \ldots, B_\ell \) be the  blocks of \( H\). By Claim~\ref{2base}, suppose \( v_1 \) is the unique vertex in \( X \) adjacent to vertices in \( H \). The following claims  can  be proved similarly as  in Section \ref{section 3}.
	\begin{claim}\label{block size2}
		For any $i \in [\ell]$, the following is true.
		\begin{itemize}
			\item [\rm (1)]$d_H\left(v_1\right) \geq p$;
			\item [\rm (2)]For any vertex $u \in V\left(B_i\right)$, we have $d_{B_i}(u) \geq p-1$;
			\item [\rm (3)]For any vertex \( u \in V(B_i)\setminus N_G(v_1) \), we have \( d_{B_i}(u) \geq p \).
		\end{itemize}
	\end{claim}
	
	\begin{claim}\label{sim}
		For any $1\leq i\neq j\leq \ell$, there are no vertices $u\in V(B_i)\setminus V(B_j)$ and  $v\in V(B_j)\setminus V(B_i)$ such that both $u$ and $v$ are adjacent to $v_1$.
	\end{claim}
	
	\begin{claim}\label{simi}
		\( H'=G[V(H)\cup\{v_1\}]\) is a strict block-cut tree, and each vertex \( v \) in \( H'\) has a degree of at least \( p \) within the block  containing \( v \).
	\end{claim}
	
	\begin{claim}
		Let \( B'_1, B'_2, \ldots, B'_\ell \) be the  blocks of \(H' .\) For any $i \in [\ell]$,  \( B_i' \) is a Hamiltonian graph with \( p+1 \leq |B_i'| \leq 2p-1 .\)
	\end{claim}
	\begin{proof}
		According to Claim \ref{simi}, we have 	$|V(B_i')|\geq p+1$ and $\delta(B_i')\geq p$ for any  $i \in [\ell]$.
		Let \( P\) be the longest path  in \( B_i'\),  $u,v$ be two ends of $P$. Since $P$ is the longest path in $B_i'$,  $d_{P}(u)\geq p$ and $d_{P}(v)\geq p$.
		By Lemma \ref{Kopylov}, \( B_i'\) contains a cycle $C$ with $|V(C)|\geq \min \left\{|V(P)|, 2p \right\}= |V(P)|$. Using Lemma \ref{Dirac}, we get
		$$
		|V(B_i')| \geq |V(C)|\geq |V(P)| \geq \min \{|V(B_i')|, 2p+1\} = |V(B_i')|,
		$$
		which implies that \( B_i' \) contains a Hamiltonian cycle $C$. Moreover,  $|V(B_i')|\leq 2p-1$ as \( B_i' \) is $C_{\geq 2p}$-free.
	\end{proof}
	
	Using the similar proof as  Claim \ref{clique}, we get the following claim.
	\begin{claim}\label{111}
		$H'$ is a single clique with  $p+1\leq |V(H')|\leq 2p-1$.
	\end{claim}

	We continue to refine the structure of  $G$ based on the number of exceptional edges in \( G[Z] \).
	\noindent
	\textbf{Case 1.} There is no exceptional edges in $G[Z]$.
	\begin{claim}\label{case1.1}
		Every even block (if exists) in $G[Z]$ is a $(2p-2)$-clique.
	\end{claim}
	\begin{proof}
		Suppose to the contrary that there exists an even block $K_x$ in  $G[Z]$ with $p\leq x\leq 2p-4$, we  modify $G$ by changing the neighborhood of a vertex $w \in Y$ to $V(K_x)\cup \{v_1\}$.
		Let \( G' \) be the resultant graph. Clearly, \( G' \) is \( \{C_{\geq 2p}, M_{s+1}\} \)-free.
		During the operation, the number of $r$-cliques increase $\binom{x+1}{r-1} -\binom{p-1}{r-1}\geq0,$ which implies that  $\mathcal{N}(K_r, G') \geq \mathcal{N}(K_r, G).$ Moreover, we have $e(G')>e(G)$, which leads to a contradiction.
	\end{proof}
	\begin{claim}\label{case1.2}
		In  $G[Z]$, there exists at most one odd block with an order smaller than $2 p-3$. Moreover, if there are multiple odd blocks and at least one even block, then all odd blocks have an order of  $2p-3$.
	\end{claim}
	\begin{proof}
		Assume that $K_x$ and $K_y$ are two maximal cliques of odd order in $G[Z]$ with $p \leq y \leq x \leq 2 p-5$. We can replace $K_x \cup K_y$ with $K_{x+2}\cup K_{y-2}$ without increasing $\nu(G)$ or creating a cycle in $C_{\geq 2p}$. Let the resultant graph be $G'$. By Lemma \ref{ch}, we have $\binom{x+2}{r}+\binom{y-2}{r}\geq \binom{x}{r}+\binom{y}{r}$. Then we have  $\mathcal{N}(K_r, G') \geq \mathcal{N}(K_r, G)$ and $e(G')>e(G)$, a contradiction. The first statement holds.
		
		For the second statement, suppose there are multiple odd blocks and at least one even block. If the smallest odd block has an order of $2 t-1$, where $\frac{p+1} {2} \leq t \leq p-2$, we claim that
		\begin{align}
			\binom{2 p-1}{r}+\binom{2 t}{r}+\binom{p-1}{r-1} \geq\binom{ 2 p-2}{r}+\binom{2 t+2}{r},\label{4}
		\end{align}
		and
		\begin{align}
			\binom{2 p-2}{r}+\binom{2 t}{r} \geq\binom{ 2 p-1}{r}+\binom{2 t-2}{r}+\binom{p-1}{r-1} .\label{5}
		\end{align}
		Otherwise,
		\begin{itemize}
			\item   if $\binom{2 p-1}{r}+\binom{2 t}{r}+\binom{p-1}{r-1} <\binom{ 2 p-2}{r}+\binom{2 t+2}{r}$, then we remove a vertex from $Y$ and replace $K_{2 p-2}\cup K_{2 t-1}$ with $K_{2 p-3} \cup K_{2 t+1}$ in $G[Z]$;
			\item  if$\binom{2 p-2}{r}+\binom{2 t}{r} <\binom{ 2 p-1}{r}+\binom{2 t-2}{r}+\binom{p-1}{r-1}$, then we add a vertex to $Y$ and replace $K_{2 p-3}\cup  K_{2 t-1}$ with $K_{2 p-2}\cup K_{2 t-3}$ in $G[Z]$.
		\end{itemize}
		Neither operation increases $\nu(G)$ or creates a cycle in $C_{\geq 2p}$, but they do increase the number of $r$-cliques, a contradiction. Combining \eqref{4} and \eqref{5} , we obtain $2\binom{2 t}{r} \geq\binom{ 2 t-2}{r}+\binom{2 t+2}{r}.$
		In particular,   $2\binom{2 t}{2} \geq\binom{ 2 t-2}{2}+\binom{2 t+2}{2},$ which contradicts $2\binom{2 t}{2}<\binom{2 t-2}{2}+$ $\binom{2 t+2}{2}$ from Lemma \ref{ch}.
	\end{proof}
	\begin{claim}\label{case1.3}
		In $G[Z]$, the number of  blocks which are \( (2p-3) \)-cliques is at most \( p-2 \) .	\end{claim}
	\begin{proof}
		Suppose to the contrary that there exist $p-1$  blocks which are \( (2p-3) \)-cliques in  $G[Z]$. We replace the copy of $(p-1)K_{2 p-3}$ in $G[Z]$ with $(p-2) K_{2 p-2}$ and add $p-1$ isolated vertices to $Y$. Denote by $G^{\prime}$ the resulting graph. During the transition from $G$ to $G^{\prime}$,  we do not increase $\nu(G)$ or create a cycle in $C_{\geq 2p}$. Moreover, we remove $(p-1)\binom{2 p-2}{r}$ $r$-cliques and add $(p-2)\binom{2 p-1}{r} +(p-1)\binom{p-1}{r-1}$ $r$-cliques. Note for $r\geq 3$, it follows that
		\begin{align*}
			\frac{(p-2)\binom{2 p-1}{r} +(p-1)\binom{p-1}{r-1}}{(p-1)\binom{2 p-2}{r}}
			>\frac{(p-2)\binom{2 p-1}{r}}{(p-1)\binom{2 p-2}{r}}
			=\frac{2 p^2-5 p+2}{2 p^2-(r+3)p+r+1}
			>1,
		\end{align*}
		where the last step is because $2 p^2-5 p+2-2 p^2+(r+3)p-r-1=(r-2)p-(r-1)>0.$\\
		For $r=2$, one can deduce
		\begin{align*}
			&(p-2)\binom{2 p-1}{2} +(p-1)\binom{p-1}{1}-(p-1)\binom{2 p-2}{2}\\
			=&(p-2)(2 p-1)( p-1)+(p-1)^2-(p-1)(p-1)(2 p-3)\\
			=&(p-1)\left((p-2)(2 p-1)+(p-1)-(p-1)(2 p-3)\right)\\
			=&(p-1)(p-2)>0.
		\end{align*}
		Then $e(G')>e(G)$, leading to a contradiction.
	\end{proof}
	
	\begin{claim}\label{case1.4}
		$\nu(G[Z])=s-p+1$.
	\end{claim}
	\begin{proof}
		Recall that $G[X \cup Y]=K_{p-1} \vee I_{n-t_0-p+1}$ from Lemma \ref{stability}. If $\nu(G[Z])<s-p+1$, then adding an edge in $Y$ increases the number of edges by 1. It is easy to verify that the resultant graph is $\{C_{\geq2 p},M_{s+1}$\}-free with more edges than $G,$ leading to a contradiction.
	\end{proof}
	
	We claim that under the condition $\nu(G) \leq s$, the number of blocks which are $(2 p-2)$-cliques in $G[Z]$ should be as large as possible. Note that  $2 p-2$ vertices  in $K_{2 p-2}\subseteq G[Z]$ contribute $\binom{2 p-1}{r}$ $r$-cliques, while in $Y$ they   contribute $(2 p-2)\binom{p-1}{r-1}$ $r$-cliques.
	For any $r\geq 3,$ we have
	$$\binom{2p-1}{r}=\sum_{i=0}^r\binom{p}{i}\binom{p-1}{r-i} \geq\binom{p}{1}\binom{p-1}{r-1} +\binom{p}{r-1}\binom{p-1}{1} \geq (2p-1)\binom{p-1}{r-1}.$$
	For $r=2,$  $\binom{2p-1}{2} > (2p-2)\binom{p-1}{1}$ holds.
	Thus under the condition $\nu(G) \leq s$, the number of blocks which are $(2 p-2)$-cliques in $G[Z]$ should be as large as possible.
	Similarly, the number of $(2 p-3)$-cliques in $G[Z]$ should also be as large as possible.
	Therefore, it is desirable to have as many independent maximal  $(2 p-2)$-cliques or maximal $(2 p-3)$-cliques in $G[Z]$ as possible.
	
	If there are no even blocks in $G[Z]$. Let \( q=\big\lfloor\frac{s-p+1}{p-2}\big\rfloor \) and \( t=s-p+1-q(p-2) \). Then we conclude that $G[Z]=q K_{2 p-3}$ if $t=0$ or $G[Z]=q K_{2 p-3}\cup K_{2 t+1}$ if $t \neq 0$ which implies $G=G_3$.
	
	Suppose that there is at least one even block in $G[Z]$. Let \( c=\big\lfloor\frac{s-p+1}{p-1}\big\rfloor \) and \( d=s-p+1-c(p-1) \).
	If there are multiple odd blocks and at least one even block, then by Claims \ref{case1.1}--\ref{case1.4}, we have $G[Z]=x K_{2 p-2} \cup y K_{2 p-3}$ for some $1\leq x$ and $2 \leq y \leq p-2$ with $x(p-1)+y(p-2)=s-p+1$. Moreover, we have $c (p-1)+d=s-p+1=x (p-1)+y(p-2)$. Solving for $x$, we get $x=c-y+\frac{d+y}{p-1}$. Since $0 \leq d \leq p-2$ , $2\leq y \leq p-2$ and $x$ is an integer, we have $\frac{d+y}{p-1}=1$. Then $y=p-d-1$ and $x=c-p+d+2$. In this case, in view of $2 \leq y=p-d -1\leq p-2$, we have $1\leq d \leq p-3$, then $G= G_5.$
	If there is at most one odd block and at least one even block in $G[Z]$, we know that either $G[Z]$ is  $c K_{2 p-2}$ if $d=0$ or $c K_{2 p-2}\cup K_{2 d+1}$  if $d\neq 0$ with $d+c (p-1)=s-p+1.$ In this case, we have $G= G_4.$
	
	\noindent
	\textbf{Case 2.} There exists an exceptional edge $uv$ in $G[Z]$.
	\begin{claim}
		Every block in $G[Z]$ is a $(2 p-2)$-clique except for the exceptional edge $uv$.
	\end{claim}
	\begin{proof}
		If there exists a block \( K_t \) with \( p \leq t \leq 2p-3 \), we modify \( G \) by changing the neighborhoods of the vertex \( v \) to \( V(K_t)\cup \{v_1\} \).
		One can readily verify that the resultant graph is \( \{C_{\geq 2p}, M_{s+1}\} \)-free. Moreover, these operations do not decrease the number of \( r \)-cliques and strictly increase the number of edges, leading to a contradiction.
	\end{proof}
	Let $s-p=c^{\prime}(p-1)+d^{\prime}, 0 \leq d^{\prime} \leq p-2$. As previously discussed, it is desirable to maximize the number of blocks in \( G[Z] \) that are \((2p-2)\)-cliques. Hence, $ G\left[Z\right]=c'K_{2 p-2}\cup K_2.$
	Then,
	$$G=K_1\vee\left(\left(K_{p-2} \vee \left(I_{n-p-1- c'(2 p-2) }\cup K_2\right)\right) \cup c' K_{2 p-2}\right),$$ where $c^{\prime}=\lfloor\frac{s-p}{p-1}\rfloor$. We show that $d^{\prime}<p-2$. Otherwise, $d^{\prime}=p-2$, we remove $u, v$ and $2 p-4$ vertices in $Y$, and add a $K_{2 p-2}$ into $c'K_{2 p-2}$ to form $(c'+1) K_{2 p-2}$ in $G[Z]$. Denote by $G^{\prime}$ the resultant graph. Obviously $G^{\prime}$ is $C_{\geq 2 p}$-free and $\nu(G')=$ $\left(c^{\prime}+1\right)(p-1)+p-1=s$. Moreover, for any $r\geq 3,$ we have
	\begin{align*}
		& \binom{2p-1}{r}-\left(2\binom{p-1}{r-1}+\binom{p-1}{r-2}+(2p-4)\binom{p-1}{r-1}\right)\\
		=&\sum_{i=0}^r\binom{p}{i}\binom{p-1}{r-i}-\binom{p-1}{r-2}-(2p-2)\binom{p-1}{r-1}\\
		\geq&\binom{p}{1}\binom{p-1}{r-1}+\binom{p}{r-1}\binom{p-1}{1}-\binom{p-1}{r-2}-(2p-2)\binom{p-1}{r-1}\\
		\geq&\binom{p-1}{r-1}+(p-2)\binom{p-1}{r-2}\geq 0.
	\end{align*}
	Therefore, $\mathcal{N}\left(K_r, G^{\prime}\right) \geq \mathcal{N}\left(K_r, G\right)$. Furthermore, it is easy to check that $e(G')>e(G)$, which leads a contradiction.
	Thus $d^{\prime}<p-2$. Note $s-p+1=c^{\prime}(p-1)+d^{\prime}+1$ and $ d^{\prime}+1<p-1$. It follows that  $c^{\prime}=c=\lfloor\frac{s-p+1}{p-1}\rfloor$ and $d=d^{\prime}+1$.
	Then
	$$
	G=K_1\vee\left(\left(K_{p-2} \vee\left(K_2 \cup I_{n-p-1-c(2p-2)}\right)\right) \cup c K_{2 p-2}\right)=G_6.
	$$
	In this case $d=d^{\prime}+1 \geq 1$.
	The result follows.

\end{document}